\documentclass[10pt]{amsart}
\usepackage{amssymb,amstext,amsmath,amscd,amsthm,amsfonts,enumerate,graphicx,latexsym,stmaryrd}
\usepackage[usenames]{color}
\usepackage[all]{xy}

\newtheorem{thm}{Theorem}[section]
\newtheorem{lem}[thm]{Lemma}
\newtheorem{prop}[thm]{Proposition}
\newtheorem{cor}[thm]{Corollary}
\theoremstyle{definition}
\newtheorem{dfn}[thm]{Definition}
\newtheorem{ques}[thm]{Question}
\newtheorem{rem}[thm]{Remark}
\newtheorem{conv}[thm]{Convention}
\newtheorem{conj}[thm]{Conjecture}

\numberwithin{equation}{thm}
\def\A{\mathrm{A}}
\def\a{\mathfrak{a}}
\def\ann{\operatorname{ann}}

\def\cl{\operatorname{Cl}}
\def\depth{\operatorname{depth}}
\def\e{\operatorname{e}}
\def\edim{\operatorname{edim}}

\def\Ext{\operatorname{Ext}}
\def\grade{\operatorname{grade}}
\def\height{\operatorname{ht}}
\def\Hom{\operatorname{Hom}}

\def\lp{\mathrm{(LP)}}
\def\m{\mathfrak{m}}

\def\Min{\operatorname{Min}}
\def\Mod{\operatorname{Mod}}
\def\n{\mathfrak{n}}
\def\p{\mathfrak{p}}
\def\Pic{\operatorname{Pic}}
\def\q{\mathfrak{q}}
\def\r{\mathrm{R}}
\def\spec{\operatorname{Spec}}
\def\st{\mathrm{(S_2)}}
\def\tr{\operatorname{tr}}
\def\Z{\mathbb{Z}}
\begin{document}
\allowdisplaybreaks
\title{Rings whose ideals are isomorphic to trace ideals}
\date{\today}
\author{Toshinori Kobayashi}
\address{Graduate School of Mathematics, Nagoya University, Furocho, Chikusaku, Nagoya, Aichi 464-8602, Japan}
\email{m16021z@math.nagoya-u.ac.jp}
\author{Ryo Takahashi}
\address{Graduate School of Mathematics, Nagoya University, Furocho, Chikusaku, Nagoya, Aichi 464-8602, Japan/Department of Mathematics, University of Kansas, Lawrence, KS 66045-7523, USA}
\email{takahashi@math.nagoya-u.ac.jp}
\urladdr{https://www.math.nagoya-u.ac.jp/~takahashi/}
\thanks{TK was partly supported by JSPS Grant-in-Aid for JSPS Fellows 18J20660.
RT was partly supported by JSPS Grant-in-Aid for Scientific Research 16K05098 and JSPS Fund for the Promotion of Joint International Research 16KK0099}
\begin{abstract}
Let $R$ be a commutative noetherian ring.
Lindo and Pande have recently posed the question asking when every ideal of $R$ is isomorphic to some trace ideal of $R$.
This paper studies this question and gives several answers.
In particular, a complete answer is given in the case where $R$ is local: it is proved in this paper that every ideal of $R$ is isomorphic to a trace ideal if and only if $R$ is an artinian Gorenstein ring, or a $1$-dimensional hypersurface with multiplicity at most $2$, or a unique factorization domain.
\end{abstract}
\maketitle
\section{Introduction}

This paper deals with trace ideals of commutative noetherian rings.
The notion of trace ideals is classical and fundamental; a lot of studies on this notion have been done in various situations.
Many references on trace ideals can be found in \cite{L,LP}.
Other than them, for instance, trace ideals play an important role in the proof of a main result of Huneke and Leuschke \cite{HL} on the Auslander--Reiten conjecture.
Recently, Goto, Isobe and Kumashiro \cite{GIK} study correspondences of trace ideals with stable ideals and finite birational extensions.

The main purpose of this paper is to consider a question on trace ideals raised by Lindo and Pande \cite{LP}.
They prove as their main result that a local ring is an artinian Gorenstein ring if and only if every ideal is a trace ideal, and ask for which ring every ideal is isomorphic to some trace ideal.
This question originates the celebrated Huneke--Wiegand conjecture: Lindo \cite{L} shows that for such a Gorenstein domain the conjecture holds.
In this paper, we begin with answering the question for rings with full of zerodivisors, which complements the result of Lindo and Pande.

\begin{thm}
Let $R$ be a commutative noetherian ring all of whose nonunits are zerodivisors (e.g. a local ring of depth $0$).
Then the following are equivalent.
\begin{enumerate}[\rm(1)]
\item
$R$ is an artinian Gorenstein ring.
\item
Every ideal of $R$ is a trace ideal.
\item
Every principal ideal of $R$ is a trace ideal.
\item
Every ideal of $R$ is isomorphic to a trace ideal.
\item
Every principal ideal of $R$ is isomorphic to a trace ideal.
\end{enumerate}
\end{thm}

Next we investigate the question of Lindo and Pande in the case of a local ring of depth one.
We prove the following theorem, which states that such a ring as in the question is nothing but a hypersurface singularity of type $(\A_n)$, under some mild assumptions.
This theorem also removes the assumption of a Gorenstein domain from Lindo's result mentioned above.

\begin{thm}
Let $R$ be a commutative noetherian local ring of depth $1$.
Consider the conditions:
\begin{enumerate}[\rm(1)]
\item
Every ideal of $R$ is isomorphic to a trace ideal,
\item
$R$ is a hypersurface with Krull dimension $1$ and multiplicity at most $2$,
\item
The completion $\widehat{R}$ of $R$ is an $(\A_n)$-singularity of Krull dimension $1$ for some $0\le n\le\infty$.
\end{enumerate}
Then the implications {\rm(1)}$\iff${\rm(2)} $\Longleftarrow$ {\rm(3)} hold.
If the residue field of $R$ is algebraically closed and has characteristic $0$, then all the three conditions are equivalent.
\end{thm}

Finally, we explore the question of Lindo and Pande in the higher-dimensional case.
It turns out that the condition in the question is closely related to factoriality of the ring.

\begin{thm}
Let $R$ be a commutative noetherian ring.
Assume that all maximal ideals of $R$ have height at least $2$.
Then the following are equivalent.
\begin{enumerate}[\rm(1)]
\item
Every ideal of $R$ is isomorphic to a trace ideal.
\item
$R$ is a product of factorial rings (i.e., unique factorization domains).
\end{enumerate}
In particular, when $R$ is local, every ideal of $R$ is isomorphic to a trace ideal if and only if $R$ is factorial.
\end{thm}

Combining all the above three theorems, we obtain the following characterization of the local rings whose ideals are isomorphic to trace ideals, which gives a complete answer to the question of Lindo and Pande for local rings.

\begin{cor}
Let $R$ be a commutative noetherian local ring.
Then the following are equivalent.
\begin{enumerate}[\rm(1)]
\item
Every ideal of $R$ is isomorphic to a trace ideal.
\item
The ring $R$ satisfies one of the following conditions.
\begin{enumerate}[\rm(a)]
\item
$R$ is an artinian Gorenstein ring.
\item
$R$ is a hypersurface of Krull dimension $1$ and multiplicity at most $2$.
\item
$R$ is a unique factorization domain.
\end{enumerate}
\end{enumerate} 
\end{cor}

This paper is organized as follows.
In Section 2, we recall the definition of trace ideals and their several basic properties. We also give a couple of observations on the Lindo--Pande condition.
In Section 3, we consider characterizing rings that satisfy the Lindo--Pande condition.
We state and prove our main results including the theorems introduced above.

\section{Trace ideals and the Lindo--Pande condition}

We begin with our convention.

\begin{conv}
Throughout this paper, unless otherwise specified, all rings are assumed to be commutative and noetherian.
Let $R$ be a ring with total quotient ring $Q$.
Colons are always taken in $Q$; when we need to consider colons in $R$, we use the notation $\ann_R$.
We denote by $(-)^*$ the $R$-dual functor $\Hom_R(-,R)$.
We use $\ell(-)$, $\e(-)$ and $\nu(-)$ to denote the length, the multiplicity and the minimal number of generators of a module, respectively.
For a local ring $(R,\m)$ we denote by $\edim R$ the embedding dimension of $R$, i.e., $\edim(R)=\nu(\m)$.
We often omit subscripts/superscripts as long as they are obvious.
\end{conv}

We recall the definition of a trace ideal.

\begin{dfn}
Let $M$ be an $R$-module.
The {\em trace} of $M$ is defined as the ideal
$$
\tr_RM=(f(x)\mid f\in M^*,\,x\in M)
$$
of $R$, that is, each element has the form $\sum_{i=1}^nf_i(x_i)$ with $f_i\in M^*$ and $x_i\in M$.
Define the $R$-linear map
$$
\lambda_M^R:M^*\otimes_RM\to R,\qquad
f\otimes x\mapsto f(x).
$$
Then $\tr_RM$ is nothing but the image of $\lambda_M^R$.
Using this, one can check that if $M,N$ are $R$-modules with $M\cong N$, then $\tr_RM=\tr_RN$.
An ideal $I$ of $R$ is called an {\em trace ideal} if $I=\tr M$ for some $R$-module $M$.
\end{dfn}

In the proofs of our results on trace ideals of $R$, it is essential to investigate $R$-submodules of $Q$ and their colons in $Q$.
We start by remarking an elementary fact, which will be used several times later.

\begin{rem}\label{fg}
Let $M$ be an $R$-submodule of $Q$.
If $M$ is finitely generated, then $M$ is isomorphic to an ideal of $R$, which can be taken to contain a non-zerodivisor of $R$ if so does $M$.
\end{rem}

\begin{proof}
Let $x_1,\dots,x_n$ generate $M$.
Write $x_i=\frac{y_i}{z_i}$ as an element of $Q$, and put $z=z_1\cdots z_n$.
Then $zM$ is an ideal of $R$.
As $z$ is a non-zerodivisor of $R$, the module $M$ is isomorphic to $zM$.
If $M$ contains a non-zerodivisor $r$ of $R$, then $zM$ contains the element $zr$, which is a non-zerodivisor of $R$.
\end{proof}

For $R$-submodules $M,N$ of $Q$ we denote by $MN$ the $R$-submodule $\langle xy\mid x\in M,\,y\in N\rangle$ of $Q$, which consists of the sums of elements of the form $xy$ with $x\in M$ and $y\in N$.
Here are several fundamental properties of colons and trace ideals, which will be used throughout the paper.

\begin{prop}\label{4}
Let $M$ be an $R$-submodule of $Q$ containing a non-zerodivisor $c$ of $R$.
\begin{enumerate}[\rm(1)]
\item
Let $N$ be an $R$-submodule of $Q$.
The assignments $f\mapsto\frac{1}{c}f(c)$ and $(x\mapsto\alpha x)\mapsfrom\alpha$ make an isomorphism $\Hom_R(M,N)\cong N:M$ of $R$-modules.
\item
There is an equality $\tr M=(R:M)M$ in $Q$.
\item
The equality $M=\tr M$ holds in $Q$ if and only if the equality $M:M=R:M$ holds in $Q$.
\item
Suppose that $M$ is finitely generated.
Then $M$ is reflexive if and only if there is an equality $M=R:(R:M)$ in $Q$.
\end{enumerate}
\end{prop}

\begin{proof}
(1) One can show that the equality $f(c)x=cf(x)$ in $Q$ holds for each $x\in M$ by describing $x$ as an element of $Q$.
It is now easy to verify that the two assignments define mutually inverse bijections.

(2) We can directly check the assertion by using the isomorphism in (1).

(3) By (2) we have only to show that $M=(R:M)M$ if and only if $M:M=R:M$.
It is obvious that $M\supseteq(R:M)M$ if and only if $M:M\supseteq R:M$.
The implications
$$
M\subseteq(R:M)M
\implies M\subseteq R
\implies M:M\subseteq R:M
\implies 1\in R:M
\implies M\subseteq(R:M)M
$$
hold, which shows that $M\subseteq(R:M)M$ if and only if $M:M\subseteq R:M$.

(4) By Remark \ref{fg} we see that $R:M$ contains a non-zerodivisor of $R$.
Applying (1) twice, we have isomorphisms $M^{**}=(M^*)^*\cong(R:M)^*\cong R:(R:M)$.
Composition with the canonical homomorphism $M\to M^{**}$ gives a homomorphism $M\to R:(R:M)$, which we observe is nothing but the inclusion map.
The assertion immediately follows from this.
\end{proof}

Lindo and Pande \cite{LP} raise the question asking when each ideal of a given ring is isomorphic to a trace ideal.
To consider this question effectively, we give a name to the condition in it.

\begin{dfn}
We define the {\em Lindo--Pande condition} $\lp$ by the following.
\begin{enumerate}
\item[$\lp$]
Every ideal of $R$ is isomorphic to some trace ideal of $R$.
\end{enumerate}
\end{dfn}

\begin{ques}[Lindo--Pande]
When does $R$ satisfy $\lp$?
\end{ques}

Let us give several remarks related to the condition $\lp$.

\begin{rem}\label{3}
\begin{enumerate}[(1)]
\item
Let $M,N$ be $R$-modules.
If $M\cong\tr N$, then $M\cong\tr M$.
Therefore, $\lp$ is equivalent to saying that each ideal $I$ of $R$ isomorphic to its trace: $I\cong\tr I$.
\item
When $R$ satisfies $\lp$, any finitely generated $R$-submodule $M$ of $Q$ admits an isomorphism $M\cong\tr M$.
\item
If $R$ satisfies $\lp$, then so does $R_S$ for each multiplicatively closed subset $S$ of $R$.
When $R$ is local, if the completion $\widehat R$ satisfies $\lp$, then so does $R$.
\end{enumerate}
\end{rem}

\begin{proof}
(1) Taking the traces of both sides of the isomorphism $M\cong\tr N$, we have $\tr M=\tr(\tr N)$.
The latter trace coincides with $\tr N$ by \cite[Proposition 2.8(iv)]{L}.
Hence $\tr M=\tr N\cong M$.

(2) The assertion follows from Remark \ref{fg} and (1).

(3) The assertion on localization is shown by using (1) and \cite[Proposition 2.8(viii)]{L}.
For the assertion on completion, apply (1) and \cite[Exercise 7.5]{E}.
\end{proof}

Now we recall that an {\it invertible} $R$-module is by definition a finitely generated $R$-module $M$ such that $M_\p\cong R_\p$ for every prime ideal $\p$ of $R$.
The isomorphism classes of invertible $R$-modules form an abelian group with multiplication $\otimes_R$ and identity $[R]$, which is called the {\em Picard group} $\Pic R$ of $R$.
The condition $\lp$ implies the triviality of this group.

\begin{prop}\label{20}
If $R$ satisfies $\lp$, then $\Pic R=0$.
\end{prop}

\begin{proof}
Let $M$ be an invertible $R$-module.
By \cite[Theorem 11.6b]{E} the $R$-module $M$ is isomorphic to an $R$-submodule of $Q$, and we get $M\cong\tr_RM$ by Remark \ref{3}(2).
By \cite[Theorem 11.6a]{E} the map $\lambda_M^R:M^*\otimes_RM\to R$ is an isomorphism, which implies $\tr_RM=R$.
Hence we obtain an isomorphism $M\cong R$, and consequently, the Picard group of $R$ is trivial.
\end{proof}

Recall that a {\em Dedekind domain} is by definition an integral domain whose nonzero ideals are invertible, or equivalently, a noetherian normal domain of Krull dimension at most one.
The above proposition yields a characterization of the Dedekind domains satisfying the Lindo--Pande condition.

\begin{cor}
A Dedekind domain satisfies $\lp$ if and only if it is a principal ideal domain.
\end{cor}

\begin{proof}
Fix a nonzero ideal $I$ of $R$.
If $R$ is a Dedekind domain satisfying $\lp$, then Proposition \ref{20} implies $I\cong R$.
Conversely, if $I\cong R$, then $\tr I=\tr R=R\cong I$.
The assertion now follows.
\end{proof}

\section{Characterization of rings satisfying the Lindo--Pande condition}

We first consider the Lindo--Pande condition for (not necessarily local) rings whose nonunits are zerodivisors.
For this, we need to extend a theorem of Lindo and Pande to non-local rings; the assertion of the following proposition is nothing but \cite[Theorem 3.5]{LP} in the case where $R$ is local.

\begin{prop}\label{13}
The following are equivalent.
\begin{enumerate}[\rm(1)]
\item
$R$ is artinian and Gorenstein.
\item
Every ideal of $R$ is a trace ideal of $R$.
\item
Every principal ideal of $R$ is a trace ideal of $R$.
\end{enumerate}
\end{prop}

\begin{proof}
Let $I$ be an ideal of $R$.
Then $I$ is a trace ideal if and only if $I=\tr_RI$ by \cite[Proposition 2.8(iv)]{L}.
In general, $I$ is contained in $\tr_RI$ by \cite[Proposition 2.8(iv)]{L} again, which enables us to define the quotient $(\tr_RI)/I$.
Using \cite[Proposition 2.8(viii)]{L}, we see that
\begin{align*}
I=\tr_RI
&\iff(\tr_RI)/I=0
\iff((\tr_RI)/I)_\p=0\text{ for all }\p\in\spec R\\
&\iff(\tr_{R_\p}I_\p)/I_\p=0\text{ for all }\p\in\spec R
\iff I_\p=\tr_{R_\p}I_\p\text{ for all }\p\in\spec R.
\end{align*}
Thus we can reduce to the local case and apply \cite[Theorem 3.5]{LP} to deduce the proposition.
\end{proof}

Using the above proposition, we obtain the following theorem including a criterion for a ring with full of zerodivisors to satisfy the Lindo--Pande condition.
Note that in the case where $R$ is local the assumption of the theorem is equivalent to the condition that $R$ has depth zero.

\begin{thm} \label{new1}
Assume that all non-zerodivisors of $R$ are units.
Then the following are equivalent.
\begin{enumerate}[\rm(1)]
\item
$R$ is artinian and Gorenstein.
\item
Every ideal of $R$ is a trace ideal of $R$.
\item
Every principal ideal of $R$ is a trace ideal of $R$.
\item
Every ideal of $R$ is isomorphic to a trace ideal of $R$, that is, $R$ satisfies $\lp$.
\item
Every principal ideal of $R$ is isomorphic to a trace ideal of $R$.
\end{enumerate}
\end{thm}

\begin{proof}
The equivalences (1)$\iff$(2)$\iff$(3) follow from Proposition \ref{13}, while the implications (2)$\implies$(4)$\implies$(5) are obvious.
It suffices to show the implication (5)$\implies$(3).

Assume that (3) does not hold, namely, that there exists a principal ideal $(x)$ of $R$ which is not a trace ideal.
Then, in particular, $x$ is nonzero.
It follows from (5), Remark \ref{3}(1) and \cite[Lemma 2.5]{LP} that $(x)\cong\tr(x)=\ann(\ann(x))$.
Let $\phi:\ann(\ann(x))\to(x)$ be the isomorphism, and $\theta:(x)\to\ann(\ann(x))$ the inclusion map.
The endomorphism $\phi\theta:(x)\to(x)$ corresponds to an endomorphism $R/\ann(x)\to R/\ann(x)$, which is a multiplication map by some element $\overline a\in R/\ann(x)$.
Then $\phi\theta$ is the multiplication map by the element $a\in R$.
Since $\phi\theta$ is injective, $a$ is a non-zerodivisor on $(x)$.
Hence $\grade((a),(x))$ is positive, or in other words, $\Hom_R(R/(a),(x))=0$.
Taking the $R$-dual of the isomorphism $(x)\cong R/\ann(x)$ yields an isomorphism $(x)^*\cong\ann(\ann(x))$, and hence $(x)\cong(x)^*$.
There are isomorphisms
$$
0=\Hom_R(R/(a),(x))\cong\Hom_R(R/(a),(x)^*)\cong(R/(a)\otimes_R(x))^*=(R/(a)+\ann(x))^*\cong\ann((a)+\ann(x)),
$$
which show that the ideal $(a)+\ann(x)$ contains a non-zerodivisor of $R$, which is a unit by the assumption of the theorem.
Therefore, $1=ab+c$ for some $b\in R$ and $c\in\ann(x)$.
We have $\phi(x)=\phi\theta(x)=ax$ and $x=(ab+c)x=abx=b\phi(x)$.
Take any element $y\in\ann(\ann(x))$.
There exists an element $d\in R$ such that $\phi(y)=dx$.
Then $\phi(y)=db\phi(x)=\phi(dbx)$, which implies $y=dbx$ as $\phi$ is injective.
Thus $y$ belongs to $(x)$.
Consequently, we obtain $(x)=\ann(\ann(x))=\tr(x)$.
This contradicts our assumption that $(x)$ is not a trace ideal.
We now conclude that (5) implies (3).
\end{proof}

Next, we study the Lindo--Pande condition for local rings of depth one.
We start by showing a lemma on Gorenstein local rings of dimension one.
Recall that a local ring $R$ is called a {\em hypersurface} if $R$ has codepth at most one, i.e., $\edim R-\depth R\le1$.
This is equivalent to saying that the completion of $R$ is isomorphic to the quotient of a regular local ring by a principal ideal.
A Cohen--Macaulay local ring is said to have {\em minimal multiplicity} if the equality $\e(R)=\edim R-\dim R+1$ holds; see \cite[Exercise 4.6.14]{BH}.

\begin{lem}\label{2}
Let $R$ be a $1$-dimensional Gorenstein local ring with maximal ideal $\m$.
If $\m\cong\m^2$, then $R$ is a hypersurface with $\e(R)\le2$.
\end{lem}

\begin{proof}
Put $s=\edim R$.
Note that $\m\cong\m^2\cong\m^3\cong\cdots$.
Hence $\nu(\m^i)=s$ for all $i>0$, and $\ell(R/\m^{n+1})=\sum_{i=0}^n\ell(\m^i/\m^{i+1})=\sum_{i=0}^n\nu(\m^i)=(n+1)s$.
Therefore $\e(R)=\lim_{n\to\infty}\frac{1}{n}\ell(R/\m^{n+1})=s$.
As $R$ has dimension one, it has minimal multiplicity.
Since $R$ is Gorenstein, it satisfies $s\le2$ (see \cite[Corollary 3.2]{Sa}) and so it is a hypersurface.
\end{proof}

We need one more lemma for our next goal, and to prove the lemma we make a remark on homomorphisms of modules over birational extensions.

\begin{rem}\label{6}
Let $S$ be a ring extension of $R$ in $Q$.
Let $M$ and $N$ be $S$-modules such that $N$ is torsion-free as an $R$-module.
Then $\Hom_S(M,N)=\Hom_R(M,N)$.
\end{rem}

\begin{proof}
Let $f:M\to N$ be an $R$-homomorphism.
Take $a\in S$ and $x\in M$.
What we want to show is that $f(ax)=af(x)$.
Write $a=\frac{b}{c}$ as an element of $Q$.
We have $c(f(ax)-af(x))=cf(ax)-caf(x)=f(cax)-caf(x)=f(bx)-bf(x)=0$.
Since $N$ is torsion-free over $R$, we get $f(ax)-af(x)=0$.
\end{proof}

\begin{lem}\label{5}
Let $I$ be a reflexive ideal of $R$ containing a non-zerodivisor of $R$, and set $S=I:I$.
Assume that the equality $I=\tr_RI$ holds.
Then one has an equality $I=\tr_R S$.
In particular, if there is an isomorphism $S\cong\tr_RS$ of $R$-modules, then one has an isomorphism $I\cong S$ of $S$-modules.
\end{lem}

\begin{proof}
First of all, note that $I$ is an $S$-module.
We apply Proposition \ref{4} several times.
We have $S=I:I=R:I$ and $I=R:(R:I)=R:S$.
Hence $\tr_RS=(R:S)S=IS=I(I:I)=I$.
Therefore, if $S\cong\tr_RS$, then there is an $R$-isomorphism $I\cong S$, and it is an $S$-isomorphism by Remark \ref{6}.
\end{proof}

For each $n\in\Z_{\ge0}\cup\{\infty\}$, a $1$-dimensional {\em hypersurface singularity of type $(\A_n)$} (or {\em $(\A_n)$-singularity} for short) is by definition a ring that is isomorphic to the quotient
$$
R_n=k[\![x,y]\!]/(x^2+y^{n+1})
$$
of a formal power series ring over a field $k$, where we set $R_0=k[\![x]\!]$ and $R_\infty=k[\![x,y]\!]/(x^2)$.
It is known that a $1$-dimensional $(\A_n)$-singularity has finite (resp. countable) Cohen--Macaulay representation type for $n\in\Z_{\ge0}$ (resp. $n=\infty$); see \cite[Corollary (9.3) and Example (6.5)]{Y}.
Hence, there exist only at most countably many indecomposable torsion-free modules over such a ring.

Now we can achieve our second purpose of this section, which is to give a characterization of the local rings of depth one that satisfy the Lindo--Pande condition.

\begin{thm}\label{new2}
Let $(R,\m,k)$ be a local ring with $\depth R=1$.
Consider the following conditions.
\begin{enumerate}[\rm(1)]
\item
The ring $R$ satisfies $\lp$.
\item
The completion $\widehat{R}$ satisfies $\lp$.
\item
The ring $R$ is a hypersurface with $\dim R=1$ and $\e(R)\le 2$.
\item
The completion $\widehat{R}$ is a $1$-dimensional $(\A_n)$-singularity for some $n\in\Z_{\ge0}\cup\{\infty\}$.
\end{enumerate}
Then the implications {\rm(1)}$\iff${\rm(2)}$\iff${\rm(3)} $\Longleftarrow$ {\rm(4)} hold.
If $k$ is algebraically closed and has characteristic $0$, then all the four conditions are equivalent.
\end{thm}

\begin{proof}
(4)$\implies$(3):
Since $\widehat R$ is a hypersurface, so is $R$.
We see directly from the definition of an $(\A_n)$-singularity that $\e(\widehat R)\le2$.
As the equality $\e(R)=\e(\widehat R)$ holds in general, we have $\e(R)\le2$.

(3)$\implies$(2):
As $\e(\widehat{R})=\e(R)$, $\dim \widehat{R}=\dim R$ and $\depth \widehat{R}=\depth R$, we may assume that $R$ is complete.
Take any ideal $I$ of $R$.
The goal is to prove $I\cong \tr I$.

We begin with the case where $I$ is an $\m$-primary ideal.
Set $S=I:I$.
Then $S$ is an intermediate ring of $R$ and $Q$ which is finitely generated as an $R$-module, and $I$ is also an ideal of $S$.
The proof of Remark \ref{fg} says $zS\subseteq R$ for some non-zerodivisor $z$ of $R$.
By \cite[Theorem 3.11]{GIK}, the ring $S$ is Gorenstein.
Using Proposition \ref{4}(1) and Remark \ref{6}, we have an $S$-isomorphism $S=I:I\to\Hom_R(I,I)=\Hom_S(I,I)$ given by $s\mapsto(i\mapsto si)$.
Hence $I$ is a {\em closed} ideal of $S$ in the sense of \cite{BV}.
It follows from \cite[Corollary 3.2]{BV} that $I$ is an invertible ideal of $S$.
As $S/\m S$ is artinian and all maximal ideals of $S$ contain $\m S$, the ring $S$ is semilocal.
We observe $I\cong S$ by \cite[Lemma 1.4.4]{BH}.
Thus it is enough to check that $S$ is isomorphic to its trace as an $R$-module.
Using Proposition \ref{4}, we obtain $\tr_RS=(R:S)S=R:S\cong S^*$.
Since $R$ is henselian, $S$ is a product of local rings: we have $S\cong S_1\times\cdots\times S_r$, where $S_i$ is local for $1\le i\le r$.
Each $S_i$ is a localization of $S$, so it is Gorenstein.
Hence $(S_i)^*\cong\omega_{S_i}\cong S_i$ for each $i$, and therefore $S^*\cong S$.
Consequently, we obtain $S\cong\tr_RS$ as desired.

Next we consider the case where $I$ is not an $\m$-primary ideal.
Then $I$ is contained in some minimal prime $\p$ of $R$.
When $I=0$, we have $I=\tr I$ and are done.
So we assume $I\ne0$, which forces $R$ not to be a domain.
By Cohen's structure theorem and the assumption that $\e(R)\le2$, we can identify $R$ with the ring $S/(f)$, where $(S,\n)$ is a $2$-dimensional regular local ring and $f$ is a reducible element in $\n^2\setminus\n^3$.
Write $f=gh$ with $g,h\in\n\setminus\n^2$.
Then $g,h$ are irreducible, and we see that $\Min R=\{gR,hR\}$ (possibly $gR=hR$).
Hence $\p$ is equal to either $gR$ or $hR$.
We also observe $\ann(gR)=hR\cong R/gR$ and $\ann(hR)=gR\cong R/hR$.
As both $R/gR$ and $R/hR$ are discrete valuation rings, any nonzero submodule of $\p$ is isomorphic to $\p$, and therefore we have only to show that $\p\cong \tr \p$.
Thanks to \cite[Corollary 2.9]{LP}, we obtain $\tr \p=\ann(\ann\p)=\p$, which particularly says $\p\cong\tr\p$.

(2)$\implies$(1):
This implication immediately follows from Remark \ref{3}(3).

(1)$\implies$(3):
We have $\m\subseteq\tr\m\subseteq R$ (see \cite[Proposition 2.8(iv)]{L}), and $\m\cong\tr\m$ by Remark \ref{3}(1).
If $\tr\m=R$, then $\m\cong R$, which means that $R$ is a discrete valuation ring, and we are done.
Thus we may assume $\m=\tr\m$.
Put $S=\m:\m$.
Proposition \ref{4}(3) implies $S=R:\m$.
Applying $(-)^*$ to the exact sequence $0 \to \m \to R \to k \to 0$ gives rise to an exact sequence $0\to R \xrightarrow{\phi} \m^* \to \Ext^1_R(k,R) \to 0$.
Note that $\Ext^1_R(k,R)\not=0$ as $\depth R=1$.
By Proposition \ref{4}(1), the map $\phi$ can be identified with the inclusion $R \subseteq S$ and so we have $R\not=S$.
Choose an element $x\in S\setminus R$ and set $X=R+Rx\subseteq S$. 
Since $\m X\subseteq \m S \subseteq R$, we have $\m \subseteq R:X$ and 
\[
\m=\m R\subseteq\m X\subseteq(R:X)X=\tr_RX\subseteq R,
\]
where the second equality follows from Proposition \ref{4}(2).
Hence $\tr_RX$ coincides with either $\m$ or $R$.
By Remark \ref{3}(2) we have $X\cong\tr_RX$.

Assume $\tr_RX=R$.
Then $X\cong R$, and we find an element $y\in X$ such that $X=Ry$.
As $1\in X$, we have $1=ay$ for some $a\in R$.
Since $y\in S$, we get $\m y\subseteq\m$, which shows $a\notin\m$.
Hence $a$ is a unit of $R$, and we observe $y\in R$.
Therefore $X=R$, and $x$ is in $R$, which contradicts the choice of $x$.

Thus we have to have $\tr_RX=\m$, and get an $R$-isomorphism $X\cong\m$.
This implies that $\m$ is generated by at most two elements as an $R$-module.
Hence
$$
1=\depth R \le \dim R\le \edim R\le 2.
$$
If $\dim R=2$, then the equality $\dim R=\edim R$ holds, which means that $R$ is a regular local ring.
In particular, $R$ is Cohen--Macaulay, and it follows that $1=\depth R=\dim R=2$, which is a contradiction.
Thus $\dim R=1$, and we have $\edim R-\dim R\le1$, namely, $R$ is a hypersurface.

It remains to prove that $R$ has multiplicity at most $2$.
According to Lemma \ref{2}, it suffices to show that $\m\cong\m^2$.
The $R$-module $S$ is isomorphic to $\tr_RS$ by Remark \ref{3}(2).
It follows from Lemma \ref{5} that $\m\cong S=R:\m$.
Using Proposition \ref{4}(2), we obtain $\m=\tr_R\m=(R:\m)\m\cong\m\m=\m^2$, as desired.
(In general, if a module $X$ is isomorphic to a module $Y$, then $\a X$ is isomorphic to $\a Y$ for an ideal $\a$.)

(3)$\implies$(4) (under the assumption that $k$ is algebraically closed and has characteristic $0$):
Again, we have $\e(\widehat R)\le2$.
Cohen's structure theorem implies that $\widehat{R}$ is isomorphic to a hypersurface of the form $k[\![x,y]\!]/(f)$ with $f\in(x,y)\setminus(x,y)^3$.
Changing variables, we can reduce to the case where $f=x$ or $f=x^2$ or $f=x^2+y^t$ with $t\in\Z_{>0}$; see (i) of \cite[Proof of (8.5)]{Y} and its preceding part.
\end{proof}

\begin{rem}
Let $R$ be a ring satisfying Theorem \ref{new2}(4).
Then each ideal of $R$ is a maximal Cohen--Macaulay $R$-module.
The isomorphism classes of indecomposable maximal Cohen--Macaulay $R$-modules are completely classified; see \cite[Proposition (5.11), (9.9) and Example (6.5)]{Y}.
The implication (4)$\implies$(1) in Theorem \ref{new2} can also be proved by using this classification (although it is rather complicated).
\end{rem}

Combining our Theorems \ref{new1} and \ref{new2}, we obtain a remarkable result.

\begin{cor}\label{new4}
The Lindo--Pande condition $\lp$ implies Serre's condition $\st$.
\end{cor}

\begin{proof}
Suppose that $R$ satisfies $\lp$.
Let $\p$ be a prime ideal of $R$.
The localization $R_\p$ also satisfies $\lp$ by Remark \ref{3}(3).
We see from Theorems \ref{new1} and \ref{new2} that $R_\p$ is Cohen--Macaulay when $\depth R_\p\le1$.
It is easy to observe from this that $R$ satisfies $\st$.
\end{proof}

Here we recall a long-standing conjecture of Huneke and Wiegand \cite{HW}.

\begin{conj}[Huneke--Wiegand]\label{hw}
Let $R$ be a $1$-dimensional local ring, and let $M$ be a torsion-free $R$-module having a rank.
If $M\otimes_RM^*$ is torsion-free, then $M$ is free.
\end{conj}

Huneke and Wiegand \cite[3.1]{HW} prove that Conjecture \ref{hw} holds for hypersurfaces.
The conjecture is widely open in general, even for ideals of complete intersection domains of codimension at least $2$.
Our theorem implies that the conjecture holds for a ring satisfying the Lindo--Pande condition.

\begin{cor}\label{30}
Let $R$ is a local ring of depth one, and suppose that $R$ satisfies $\lp$.
Let $M$ be an $R$-module having a rank.
If $M\otimes_RM^*$ is torsion-free, then $M$ is free.
In particular, Conjecture \ref{hw} holds for a ring satisfying $\lp$.
\end{cor}

\begin{proof}
Theorem \ref{new2} implies that $R$ is a $1$-dimensional hypersurface.
By \cite[Theorem 3.1]{HW}, both $M$ and $M^*$ are torsion-free, and either of them is free.
If $M^*$ is free, then so is $M$ by \cite[Lemma 2.13]{hwcn}.
\end{proof}

\begin{rem}
Lindo \cite[Proposition 6.8]{L} shows that Conjecture \ref{hw} holds if $R$ is a Gorenstein domain and $M$ is isomorphic to a trace ideal.
Note that our Corollary \ref{30} does not assume that $R$ is Gorenstein, or $R$ is a domain, or $M$ is torsion-free.
\end{rem}

Our next goal is to study the Lindo--Pande condition for rings having Krull dimension at least two.
The following proposition characterize the ideals of normal rings that are isomorphic to trace ideals.

\begin{prop}\label{9}
Let $M$ be a finitely generated $R$-submodule of $Q$ containing a non-zerodivisor of $R$.
Consider the following conditions.
\begin{enumerate}[\rm(1)]
\item
$M$ is isomorphic to a trace ideal of $R$.
\item
$M^*$ is isomorphic to $R$.
\item
$M$ is isomorphic to an ideal $I$ of $R$ with $\grade I\ge2$ (i.e. $\Ext_R^i(R/I,R)=0$ for $i<2$).
\end{enumerate}
Then the implications {\rm(1)} $\Longleftarrow$ {\rm(2)} $\Longleftrightarrow$ {\rm(3)} hold.
All the three conditions are equivalent if $R$ is normal.
\end{prop}

\begin{proof}
In view of Remark \ref{fg}, we can replace $M$ with an ideal $J$ of $R$ containing a non-zerodivisor.

(3)$\implies$(2):
Dualizing the natural short exact sequence $0\to I\to R\to R/I\to0$ by $R$ induces $I^*\cong R$.

(2)$\implies$(1):
Using Proposition \ref{4}(1)(2), we have $R:J\cong J^*\cong R$, and $\tr J=(R:J)J\cong RJ=J$.

(2)$\implies$(3):
If $J=R$, then we have $\grade J=\infty\ge2$ and are done.
Let $J\ne R$.
Then $(R/J)^*=0$, and dualizing the natural exact sequence $0\to J\to R\to R/J\to0$ gives an exact sequence $0\to R\to J^*\to\Ext_R^1(R/J,R)\to0$.
Combining this with the isomorphism $J^*\cong R$, we find a non-zerodivisor $x_1$ of $R$ such that $\Ext_R^1(R/J,R)\cong R/(x_1)$.
As $J$ annihilates the Ext module, it is contained in the ideal $(x_1)$.
Hence we find an ideal $J_1$ of $R$ such that $J=x_1J_1$.
It is easy to see that $J_1$ also contains a non-zerodivisor of $R$.
As $J_1$ is isomorphic to $J$, we have $J_1^*\cong R$.
Thus the argument for $J$ applies to $J_1$.
If $J_1=R$, then we are done.
If $J_1\ne R$, then we find an ideal $J_2$ and a non-zerodivisor $x_2$ with $J_1=x_2J_2$.
Iterate this procedure, and consider the case where we get ideals $J_i$ and non-zerodivisors $x_i$ such that $J_i=x_{i+1}J_{i+1}$ for all $i\ge0$.
In this case, there is a filtration of ideals of $R$:
$$
J=:J_0\subseteq J_1\subseteq J_2\subseteq J_3\subseteq \cdots.
$$
As $R$ is noetherian, this stabilizes: there exists an integer $t\ge0$ such that $J_t=J_{t+1}$.
Hence $J_{t+1}=x_{t+1}J_{t+1}$, and Nakayama's lemma gives rise to an element $r\in R$ such that $1-r\in(x_{t+1})$ and $rJ_{t+1}=0$.
The fact that $J_{t+1}$ contains a non-zerodivisor forces $r$ to be zero, and $x_{t+1}$ is a unit of $R$.
Therefore $\Ext_R^1(R/J_t,R)\cong R/(x_{t+1})=0$, and thus $\grade J_t\ge2$.
It remains to note that $J$ is isomorphic to $J_t$.

Finally, we prove the implication (1)$\implies$(2) under the additional assumption that $R$ is normal.
By Remark \ref{3}(1) the ideal $J$ is isomorphic to its trace $I:=\tr J$.
As $J\subseteq I$ by \cite[Proposition 2.8(iv)]{L}, the ideal $I$ contains a non-zerodivisor of $R$.
We have $\tr I=\tr(\tr J)=\tr J=I$ by \cite[Proposition 2.8(iv)]{L} again.
Using (1) and (3) of Proposition \ref{4}, we get $I^*\cong R:I=I:I$.
The ring $I:I$ is a module-finite extension of $R$ in $Q$, and hence it is integral over $R$.
Since $R$ is normal, we have $I:I=R$.
We thus obtain $J^*\cong I^*\cong I:I=R$.
\end{proof}

The above proposition yields a characterization of the normal domains that satisfy the Lindo--Pande condition.
For a normal domain $R$ we denote by $\cl(R)$ the {\em divisor class group} of $R$.

\begin{cor}\label{10}
A ring $R$ is a normal domain satisfying $\lp$ if and only if it is factorial.
\end{cor}

\begin{proof}
Let $R$ be a normal domain.
Then it follows from \cite[Proposition 6.1]{F} that $R$ is factorial if and only if $\cl(R)=0$.
The zero ideal is a trace ideal as $0=\tr0$.
Applying Proposition \ref{9}, we observe that $R$ satisfies $\lp$ if and only if $I^*\cong R$ for all ideals $I\ne0$.
Therefore we have only to show the following two statements (see \cite[2.10]{S}).
\begin{enumerate}[(a)]
\item
Suppose that $I^*$ is isomorphic to $R$ for every nonzero ideal $I$ of $R$.
Let $M$ be a finitely generated reflexive $R$-module of rank one.
Then $[M]=0$ in $\cl(R)$.
\item
Let $I$ be a nonzero ideal of $R$.
Then $I^*$ is a reflexive module of rank one.
\end{enumerate}

(a): As $M$ has rank $1$ and is torsion-free, it is isomorphic to an ideal $I\ne0$ of $R$.
Then $I^*$ is isomorphic to $R$ by assumption, and we get isomorphisms $M\cong M^{**}\cong I^{**}\cong R^*\cong R$.
Hence $[M]=0$ in $\cl(R)$.

(b): The module $I$ has rank $1$, and so does $I^*$.
For each $R$-module $X$, denote by $\rho(X)$ the canonical homomorphism $X\to X^{**}$.
We can directly verify that the composition $(\rho(I))^*\circ\rho(I^*)$ is the identity map of $I^*$.
Hence $I^{***}\cong I^*\oplus E$ for some $R$-module $E$.
Comparing the ranks, we see that $E$ is torsion.
As $E$ is isomorphic to a submodule of the torsion-free module $I^{***}$, it is zero.
Therefore $I^*$ is reflexive.
\end{proof}

What we want to do next is to remove from the above corollary the assumption that $R$ is a normal domain.
For this, we need to investigate the Lindo--Pande condition for a finite product of rings.

\begin{lem} \label{new3}
Let $R_1,\dots R_n$ be rings.
Then the product ring $R_1\times\cdots\times R_n$ satisfies $\lp$ if and only if $R_i$ satisfies $\lp$ for all $1\le i\le n$.
\end{lem}

\begin{proof}
The assignment $(M_1,\dots,M_n)\mapsto M_1\times\cdots\times M_n$ gives an equivalence
$$
\textstyle\prod_{i=1}^n(\Mod R_i)\cong\Mod(\prod_{i=1}^nR_i)
$$
as tensor abelian categories, where for a ring $A$ we denote by $\Mod A$ the category of arbitrary $A$-modules.
In particular, we can do the identification
$$
\textstyle\prod\Hom_{R_i}(M_i,N_i)=\Hom_{\prod R_i}(\prod M_i,\prod N_i),\qquad
\prod(M_i\otimes_{R_i}N_i)=\prod M_i\otimes_{\prod R_i}\prod N_i.
$$
Now it is easy to see that for all ideals $I_i$ of $R_i$ with $1\le i\le n$, one has
\begin{equation}\label{18}
\textstyle\tr_{\prod R_i}(\prod I_i)=\prod\tr_{R_i}I_i.
\end{equation}
The ``if'' part of the lemma directly follows from \eqref{18} (see Remark \ref{3}(1)).
Applying \eqref{18} to the ideal $0\times\cdots\times0\times I_i\times0\times\cdots\times0$ of $\prod R_i$ shows the ``only if'' part.
\end{proof}

Now we have reached our third (final) goal of this section, which is to give a criterion for a certain class of rings with Krull dimension at least two to satisfy the Lindo--Pande condition.

\begin{thm}\label{7}
Assume that all maximal ideals of $R$ have height at least $2$.
Then $R$ satisfies $\lp$ if and only if $R$ is a 
 product of factorial rings.
In particular, when $R$ is a local ring or an integral domain, it satisfies $\lp$ if and only if it is factorial.
\end{thm}

\begin{proof}
The ``if'' part follows from Corollary \ref{10} and Lemma \ref{new3}.
To prove the ``only if'' part, it suffices to show that $R$ is normal.
Indeed, suppose that it is done.
Then $R$ is a product $R_1\times \cdots\times R_n$ of normal domains; see \cite[Page 64, Remark]{Mats}.
By Lemma \ref{new3} and Corollary \ref{10}, each $R_i$ is factorial, and the proof is completed.

So let us show that $R$ is normal.
As $R$ satisfies $\st$ by Corollary \ref{new4}, it is enough to verify that $R$ satisfies $(\r_1)$.
Fix a prime ideal $\p$ of $R$ with $\height\p\le1$.
What we want to show is that $R_\p$ is a regular local ring.
By assumption, $\p$ is not a maximal ideal, and we find a prime ideal $\q$ containing $\p$ with $\height\q/\p=1$.

(i) We begin with considering the case where $\height\p=1$.
In this case, $\height\q\ge2$.
Note that $\st$ localizes, and so does $\lp$ by Remark \ref{3}(3).
Replacing $R$ with $R_\q$, we may assume that $(R,\m)$ is a local ring with $\dim R=\height\m\ge2$ and $\dim R/\p=\height\m/\p=1$.
Then $R/\p$ is a $1$-dimensional Cohen--Macaulay local ring.
Since $R$ satisfies $\st$, we have $\depth R\ge2$ and $\p$ contains a non-zerodivisor of $R$; see \cite[Proposition 1.2.10(a)]{BH}.
To show that $R_\p$ is regular, it suffices to prove that $R_\p$ has embedding dimension at most one.

Let us consider the case $\tr\p=R$.
Then $\p$ contains a nonzero free summand; see \cite[Proposition 2.8(iii)]{L}.
We find a non-zerodivisor $x$ of $R$ in $\p$ and a subideal $I$ of $\p$ such that $\p=(x)\oplus I$.
Since $(x)\cap I=0$, we have $xI=0$, and $I=0$ as $x$ is a non-zerodivisor.
Thus $\p=(x)$.
In particular, we have $\edim R_\p\le1$, which is what we want.
Consequently, we may assume that $\tr\p$ is a proper ideal of $R$.

We claim $\p=\tr\p$.
Indeed, $\tr\p$ contains $\p$ by \cite[Proposition 2.8(iv)]{L}.
Suppose that the containment is strict.
Then $\tr\p$ is $\m$-primary as $\height\m/\p=1$.
Apply the depth lemma to the natural exact sequences
$$
0\to\p\to R\to R/\p\to0,\qquad
0\to\tr\p\to R\to R/\tr\p\to0.
$$
We observe $\depth\p=2$ and $\depth(\tr\p)=1$.
Our assumption that $R$ satisfies $\lp$ and Remark \ref{3}(1) imply that $\p\cong\tr\p$, which gives a contradiction.
Thus the claim follows.

Next, we claim that $\p$ is reflexive.
In fact, let $P$ be a prime ideal of $R$.
According to \cite[Proposition 1.4.1]{BH}, it is enough to check the following.
\begin{enumerate}[\quad(a)]
\item
If $\depth R_P\le1$, then $\p R_P$ is a reflexive $R_P$-module.
\item
If $\depth R_P\ge2$, then $\depth\p R_P\ge2$.
\end{enumerate}
If $P$ does not contain $\p$, then $\p R_P=R_P$.
If $P$ contains $\p$, then $P$ coincides with $\p$ or $\m$ as $\height\m/\p=1$.
Recall that $\depth R\ge2$ and $\depth\p=2$.
The fact that $R$ satisfies $\st$ especially says $\depth R_\p=1$.
Theorem \ref{new2} and Remark \ref{3}(3) imply that $R_\p$ is a Gorenstein local ring of dimension $1$, whence $\p R_\p$ is a reflexive $R_\p$-module.
We now easily see that (a) and (b) hold, and the claim follows.

Set $S=\p:\p$.
Then $S=R:\p$ by the above first claim and Proposition \ref{4}(3).
It follows from the condition $\lp$, Remark \ref{3}(2), Lemma \ref{5} and the above two claims that $S\cong\tr_RS=\p$.
Thus we obtain an $S$-isomorphism $\p\cong S$; see Remark \ref{6}.
The ideal $\p$ contains a non-zerodivisor $x$ of $S$ such that $\p=xS$.
Note that $x$ is also a non-zerodivisor of $R$.
If $\p=xR$, then $\edim R_\p\le1$ and we are done.

Now, let us suppose that $\p\ne xR$, and derive a contradiction.
Krull's intersection theorem shows $\bigcap_{i>0}(\m^i+xR)=xR$, which implies $\p\nsubseteq\m^t+xR$ for some $t>0$.
Put $I=\p\cap(\m^t+xR)$.
Notice that $I$ contains the non-zerodivisor $x$ of $R$ and is strictly contained in $\p$.

We claim that $\p=\tr I$.
Indeed, we have
$$
\tr I=(R:I)I=(R:\p)I\subseteq(R:\p)\p=\tr\p=\p=xS=(R:\p)x\subseteq(R:\p)I=\tr I.
$$
Here, the first and third equalities follow from Proposition \ref{4}(2).
Consider the exact sequence $0\to I\xrightarrow{f}\p\to\p/I\to0$, where $f$ is the inclusion map.
Note that $\p/I$ has finite length.
As $\depth R\ge2$, the map $f^*:\p^*\to I^*$ is an isomorphism.
It is observed from this and Proposition \ref{4}(1) that the equality $(R:I)I=(R:\p)I$ appearing above holds.

This claim, the condition $\lp$ and Remark \ref{3}(1) imply $\p\cong I$.
Applying the depth lemma to the exact sequence $0\to I\to\p\to\p/I\to0$ shows that $I$ has depth $1$ as an $R$-module.
However, the $R$-module $\p$ has depth $2$, and we obtain a desired contradiction.
Thus, the proof is completed in the case $\height\p=1$.

(ii) Now we consider the case where $\height\p=0$.
We have $\height\q\ge1$.
If $\height\q=1$, then $R_\q$ is regular by (i), and so is $R_\p=(R_\q)_{\p R_\q}$, which is what we want.
Assume $\height\q\ge2$, and let us derive a contradiction.
As in (i), replacing $R$ with $R_\q$, we may assume that $(R,\m)$ is a local ring with $\depth R\ge2$ and $R/\p$ is a Cohen--Macaulay local ring of dimension $1$.
Choosing an element $y\in\m\setminus\p$, we get an exact sequence $0\to R/\p\xrightarrow{y} R/\p\to R/\p+(y)\to0$.
Since $R/\p+(y)$ has finite length and $R$ has depth at least two, taking the $R$-dual yields the isomorphism $(R/\p)^*\xrightarrow{y}(R/\p)^*$.
Nakayama's lemma implies $(R/\p)^*=0$.
Hence $\p$ has positive grade, but this contradicts the fact that $\p$ is a minimal prime.
\end{proof}

As an application of the above theorem, we observe that the Lindo--Pande condition does not necessarily ascend along the completion map $R\to\widehat R$ for a local ring $R$.

\begin{cor}
Let $R$ be a local ring.
If $\widehat R$ satisfies $\lp$, then so does $R$.
The converse also holds if $\depth R\leq 1$, but does not necessarily hold if $\depth R\ge2$.
\end{cor}

\begin{proof}
The descent of $\lp$ is included in Remark \ref{3}(3), while the ascent for $\depth R\le1$ is observed from Theorems \ref{new1} and \ref{new2}.
There exists a factorial local ring $R$ of depth $2$ whose completion is not factorial.
In fact, Ogoma's famous example \cite{Og} of a $2$-dimensional factorial local ring without a canonical module is such a ring by \cite[Corollaries 3.3.8 and 3.3.19]{BH}; see also \cite[Example 6.1]{N} and \cite[Page 145]{BH}.
Theorem \ref{7} implies that this ring $R$ satisfies $\lp$ but $\widehat{R}$ does not.
\end{proof}


\end{document}